\documentclass{article}%
\usepackage{amsfonts}
\usepackage{amsmath}
\usepackage{amssymb}
\usepackage{graphicx}%
\newtheorem{theorem}{Theorem}
\newtheorem{lemma}{Lemma}
\newtheorem{corollary}[lemma]{Corollary}

\newenvironment{proof}[1][Proof]{\noindent\textbf{#1.} }{\ \rule{0.5em}{0.5em}}
\newcommand{\nc}{\newcommand}
\nc{\cdm}{\textrm{cdm}}
\nc{\card}{\textrm{card}}
\nc{\diam}{\textrm{diam}}
\nc{\dis}{\textrm{dis}}
\nc{\ie}{i.e.\ }
\nc{\eg}{e.g.\ }
\newlength{\malong}

\nc{\nph}[1]{\settowidth{\malong}{#1}\hspace{- \malong}}
\nc{\ph}[1]{\settowidth{\malong}{$#1$}\hspace{\malong}}
\nc{\phanto}[1]{#1}
\begin{document}

\author{Jo\"{e}l Rouyer}
\title{Generic Properties of Compact Metric Spaces}
\maketitle

\begin{abstract}
We prove that there is a residual subset of the Gromov-Hausdorff space (\ie
the space of all compact metric spaces up to isometry endowed with the
Gromov-Hausdorff distance) whose points enjoy several unexpected properties.
In particular, they have zero lower box dimension and infinite upper box dimension.

\end{abstract}

\section{Introduction}

It is rather known that a `generic' continuous real function (or a `typical'
one, or `most' of them) admits a derivative at no point. This is often stated
as a kind of curiosity, for such a function is not so easy to exhibit, or even
to fancy. The aim of this paper is to give some properties of a generic
compact metric space.

When we say generic, we refer to the notion of Baire categories. We recall
that a subset of a topological space $B$ is said to be \emph{rare} if the
interior of its closure is empty. It is said to be \emph{meager}, or \emph{of
first category}, if it is a countable union of rare subsets of $B$. The space
$B$ is called a \emph{Baire space} if each meager subset of $B$ has empty
interior. The Baire's theorem states that any complete metric space is a Baire
space. The complementary of a meager subset of a Baire space is said to be
\emph{residual}. At last, given a Baire space $B$, we say that a
\emph{generic} element of $B$ enjoys a property if the set of elements which
satisfy this property is residual.

In order to state the results of this article we need a few definitions. We
say that a metric space $X$ is \emph{totally anisometric} if two distinct
pairs of points have distinct distances. We say that three points $x,y,z\in X$
are \emph{collinear} if one of the three distances between them equals the sum
of the two others. Of course, this definition matches the classical one in a
Euclidean space. A \emph{perfect set} is a closed set without isolated point.
The definitions of the upper and lower box dimensions are recalled in the
Section \ref{SD}.

We will prove that a generic compact metric space $X$:

\begin{enumerate}
\item \label{PT1}is totally discontinuous (Theorem \ref{T1}).

\item \label{PT2}is totally anisometric (Theorem \ref{T2}).

\item \label{PT3}has no collinear triples of (pairwise distinct) points
(Theorem \ref{T3}).

\item \label{PT6}is perfect (Theorem \ref{T6}).

\item \label{PC6}is homeomorphic to the Cantor set (Corollary \ref{C6}).

\item \label{PT4}admits a set of distance values $\left\{  d\left(
x,y\right)  |x,y\in X\right\}  $ which is homeomorphic to the Cantor set.
(Theorem \ref{T4}).

\item \label{PT7}cannot be embedded in any Hilbert space (Theorem \ref{T7}).

\item \label{PT5a}has zero Hausdorff and lower box dimensions (Theorem
\ref{T5}).

\item \label{PT5b}has infinite upper box dimension (Theorem \ref{T5}).
\end{enumerate}

Earlier studies were performed on a generic compact subset of some fixed (but
without special assumption on it) complete metric space, and on a generic
compact subset of $\mathbb{R}^{n}$. In particular A. V. Kuz'minykh proved that
the points \ref{PT1} and \ref{PT2} hold in this latter case \cite{Kuz}. He
also proves that any $n+1$ points of a generic compact subset of
$\mathbb{R}^{n}$ are always the vertices of a simplex, or in other words, are
never co-hyperplanar. The point \ref{PT3} can be seen as a weaker version of
this result. In the same article, the author proved that the set of distance
values of a generic compact of $\mathbb{R}^{n}$ is totally discontinuous.
Point \ref{PT4} improved his statement in our framework. Concerning the notion
of dimension, P. M. Gruber proved that the point \ref{PT5a} holds for a
generic compact subset of any fixed complete metric space $X$. He also proved
that, if $X$ has some suitable property (\eg$X=\mathbb{R}^{n}$), then a
generic compact subset of $X$ has an upper box dimension greater than or equal
to $n$.

Many other properties are given by these authors and some others
(\cite{Z3}\allowbreak\cite{Wiea}\allowbreak\cite{Blasi}\allowbreak
\cite{Z1}\allowbreak\cite{Myjak2}\allowbreak\cite{Z2}). However those
properties (e.g. to be porous \cite{Z3}, or to look dense \cite{Z2}) involve
the embedding of the generic subset in the whole space, and so, admit no
counterpart in our framework.

\section{\label{S2}The Gromov-Hausdorff space}

The section is devoted to recall the definition and the properties of the
so-called Gromov-Hausdorff space, \ie the space of all isometry classes of
compact metric spaces. We will use the same letter to designate both a metric
space (\ie a set endowed with a distance) and its underlying set. If $X$ is a
metric space, we denote by $d^{X}$ its metric. If $A$ is a part of $X $ and
$\rho$ is a positive number, we denote by $A+\rho$ the $\rho$-neighborhood of
$A$, namely%
\[
A+\rho=\left\{  x\in X|\exists y\in X~d^{X}\left(  X,Y\right)  <\rho\right\}
\text{.}%
\]
If $A=\left\{  a\right\}  $ is a singleton, we denote by $B\left(
a,\rho\right)  =A+\rho$ the open ball. We recall that the Hausdorff distance
$d_{H}^{X}\left(  A,B\right)  $ between two nonempty closed subsets $A$ and
$B$ of $X$ is the infimum of those numbers $\rho$ such that $A\subset B+\rho$
and $B\subset A+\rho$. It is easy to see that $d_{H}^{X}$ is a distance on the
set $\mathfrak{M}\left(  X\right)  $ of all nonempty compact subsets of $X$.
Moreover, if $X$ is complete (resp. compact), then $\mathfrak{M}\left(
X\right)  $ is complete (resp. compact) too \cite[p. 253]{BBI}.

Let $X$ and $Y$ be two compact metric spaces. The Gromov-Hausdorff distance
between them is defined by%
\[
d_{GH}\left(  X,Y\right)  =\inf d_{H}^{Z}\left(  f\left(  X\right)  ,g\left(
Y\right)  \right)  \text{,}%
\]
where the infimum is taken over all metric spaces $Z$ and all isometric
injections $f:X\rightarrow Z$ and $g:Y\rightarrow Z$. It is well-known that
$d_{GH}$ is a distance on the Gromov-Hausdorff space $\mathfrak{M}$ of all
compact metric space up to isometry \cite[p. 259]{BBI}.

A critical point for our purpose is the completeness of $\mathfrak{M}$. S. C.
Ferry considers the fact as `well-known' \cite{Ferry2}\cite{Ferry1}, but the
invoked references (\cite{Grove}\cite{IR1}) actually do not mention
completeness. In \cite[chap. 4]{LNRTree} the author seems unaware of the fact,
for he took the trouble to prove that some particular closed subset of
$\mathfrak{M}$ is complete. Indeed, it appears that the proof is not so easy
to seek in the literature. Since the fact is central in this paper, we chose
to include its proof at the end of the Section.

A \emph{correspondence} $R$ between two metric spaces $X$ and $Y$ is a
relation (\ie a subset of $X\times Y$) such that each element of $X$ is in
relation with at least one element of $Y$, and conversely, each element of $Y
$ is in relation with at least one element in $X$. For $x\in X$ and $y\in Y $,
we write $xRy$ instead of $\left(  x,y\right)  \in R$. The \emph{distortion}
of a correspondence $R$ between $X$ and $Y$ is the number%
\[
\dis\left(  R\right)  =\sup\left\{  \left.  \left\vert d^{X}\left(
x,x^{\prime}\right)  -d^{Y}\left(  y,y^{\prime}\right)  \right\vert
~\right\vert x,x^{\prime}\in X~y,y^{\prime}\in Y~xRy~x^{\prime}Ry^{\prime
}\right\}  \text{.}%
\]
The above notion is useful for it provides another way to compute the
Gromov-Hausdorff distance.

\begin{lemma}
\label{Ldis}\emph{\cite[p. 257]{BBI}}%
\[
d_{GH}\left(  X,Y\right)  =\frac{1}{2}\inf_{R}\dis\left(  R\right)
\]
where the infimum is taken over all correspondences $R$ between $X$ and $Y$.
\end{lemma}

Another useful result is the following

\begin{lemma}
\label{LPS}Let $\left(  X_{n}\right)  _{n\in\mathbb{N}}$ be a converging
sequence of elements of $\mathfrak{M}$, and denote by $Y$ its limit. Let
$\left(  \varepsilon_{n}\right)  _{n\in\mathbb{N}}$ a sequence of positive
numbers. Then, there exist a compact metric space $Z$, an isometric embedding
$g:Y\rightarrow Z$, and for each positive integer $n$, an isometric embedding
$f_{n}:X_{n}\rightarrow Z$, such that $d_{H}^{Z}\left(  f_{n}\left(
X_{n}\right)  ,g\left(  Y\right)  \right)  <d_{GH}\left(  X_{n},Y\right)
+\varepsilon_{n}$.
\end{lemma}

\begin{proof}
First, we can assume without loss of generality that $\varepsilon_{n}$
converges to $0$. For each integer $n$ there exists a compact metric space
$Z_{n}$ and two isometric injections $f_{n}^{\prime}:X_{n}\rightarrow Z_{n}$
and $g_{n}^{\prime}:Y\rightarrow Z_{n}$ such that
\[
d_{H}^{Z_{n}}\left(  f_{n}^{\prime}\left(  X_{n}\right)  ,g_{n}^{\prime
}\left(  Y\right)  \right)  <d_{GH}\left(  X_{n},Y\right)  +\varepsilon
_{n}\text{.}%
\]
We also assume that $Z_{n}$ is minimum for inclusion, that is $Z_{n}%
=f_{n}^{\prime}\left(  X_{n}\right)  \cup g_{n}^{\prime}\left(  Y\right)  $.
Let $Z^{\prime}=\coprod_{n\in\mathbb{N}}Z_{n}$, endowed with the
pseudo-distance%
\begin{align*}
d^{Z^{\prime}}\left(  a,b\right)   &  =d^{Z_{n}}\left(  a,b\right)
\text{,}\hspace{1.5cm}\text{if }a,b\in Z_{n}\\
d^{Z^{\prime}}\left(  a,b\right)   &  =\min_{y\in Y}\left(  d^{Z_{n}}\left(
a,g_{n}^{\prime}\left(  y\right)  \right)  +d^{Z_{m}}\left(  g_{m}^{\prime
}\left(  y\right)  ,b\right)  \right)  \text{,}\\
&  \phantom{d^{Z_{n}}\left( a,b\right)\text{,}}\hspace{1.5cm}\text{if }a\in
Z_{n}~\text{and }b\in Z_{m}~\text{with }m\neq n\text{.}%
\end{align*}
We have call $d^{Z^{\prime}}$ a pseudo-distance, but of course, this should to
be checked. Since the verification is straightforward, it is left to the
reader. Let $Z$ be the quotient of $Z^{\prime}$ by the relation of equivalence
$\sim$%
\[
a\sim b\Longleftrightarrow d^{Z^{\prime}}\left(  a,b\right)  =0\text{,}%
\]
and let $\pi:Z^{\prime}\rightarrow Z$ be the canonical surjection. We define
$f_{n}=\pi\circ f_{n}^{\prime}$ and $g=\pi\circ g_{n}^{\prime}$ ($g$ does not
depends on $n$, since $d\left(  g_{n}\left(  y\right)  ,g_{m}\left(  y\right)
\right)  =0$). It is clear that $f_{n}$ and $g$ are isometric embeddings, and
that
\[
d_{H}^{Z}\left(  f_{n}\left(  X_{n}\right)  ,g\left(  Y\right)  \right)
<d_{GH}\left(  X_{n},Y\right)  +\varepsilon_{n}\text{.}%
\]

It remains to prove that $Z$ is compact. Let $\left(  z_{k}\right)
_{k\in\mathbb{N}}$ be a sequence in $Z$. Either there exists an integer $n$
such that all but a finite number of terms of $\left(  z_{k}\right)
_{k\in\mathbb{N}}$ belong to $\pi\left(  Z_{n}\right)  $, or one can extract
from $\left(  z_{k}\right)  _{k\in\mathbb{N}}$ a subsequence such that
$z_{k}\in Z_{v\left(  k\right)  }$, where $v:\mathbb{N\rightarrow}\mathbb{N}$
is increasing. In the former case, since $Z_{n}$ is compact, we can extract
from $\left(  z_{k}\right)  $ a converging subsequence, and the proof is over.
In the latter case, there exists a sequence $\left(  y_{k}\right)
_{k\in\mathbb{N}}$ of points of $Y$ such that $d^{Z_{v\left(  k\right)  }%
}\left(  z_{k},g_{v\left(  k\right)  }\left(  y_{k}\right)  \right)  \leq
d_{GH}\left(  X_{v\left(  k\right)  },Y\right)  +\varepsilon_{v\left(
k\right)  }$. By extracting a suitable subsequence we may assume that $y_{k}$
is converging to some point $y\in Y$. It follows that%
\begin{align*}
d^{Z}\left(  z_{k},g\left(  y\right)  \right)   &  \leq d^{Z_{v\left(
k\right)  }}\left(  z_{k},g_{v\left(  k\right)  }\left(  y_{k}\right)
\right)  +d^{Z_{v\left(  k\right)  }}\left(  g_{v\left(  k\right)  }\left(
y_{k}\right)  ,g\left(  y\right)  \right) \\
&  \leq d_{GH}\left(  X_{v\left(  k\right)  },Y\right)  +\varepsilon_{v\left(
k\right)  }+d^{Y}\left(  y_{k},y\right)  \rightarrow0\text{.}%
\end{align*}
whence $\left(  z_{k}\right)  _{k\in\mathbb{N}}$ is converging. Hence $Z$ is compact.
\end{proof}

\bigskip

For $X\in\mathfrak{M}$ and $\varepsilon>0$, we denote by
\[
N\left(  X,\varepsilon\right)  =\min\left\{  \card\left(  F\right)  |F\subset
X~\forall x\in X~d\left(  x,F\right)  \leq\varepsilon\right\}  \text{,}%
\]
the minimum number of closed balls of radius $\varepsilon$ which are required
to cover $X$. It is easy to see that, for a given space $X$, the function
$N\left(  X,\bullet\right)  $ is non increasing and left-continuous. A subset
$P\subset\mathfrak{M}$ is said to be \emph{uniformly totally bounded} if on
the one hand the set of diameters of elements of $P$ is bounded, and on the
other hand, for each number $\varepsilon>0$, $\sup_{X\in P}N\left(
X,\varepsilon\right)  <\infty$. The Gromov's compactness theorem states that
every uniformly totally bounded subset of $\mathfrak{M}$ is relatively compact
\cite[p. 263]{BBI}. The completeness of $\mathfrak{M}$ follows from Gromov's theorem.

\addtocounter{theorem}{-1}

\begin{theorem}
$\mathfrak{M}$ is complete.
\end{theorem}

\begin{proof}
Let $\left(  X_{n}\right)  _{n\in\mathbb{N}}$ be Cauchy sequence of compact
metric spaces. Since the diameter function $\diam:\mathfrak{M}\rightarrow
\mathbb{R}$ is Lipschitz continuous, $\left(  \diam\left(  X_{n}\right)
\right)  _{n\in\mathbb{N}}$ is Cauchy, and so is bounded.

We claim that for $X,Y\in\mathfrak{M}$ and $\rho>0$ we have%
\[
N\left(  X,\rho+2d_{GH}\left(  X,Y\right)  \right)  \leq N\left(
Y,\rho\right)  \text{.}%
\]
Put $\varepsilon=d_{GH}\left(  X,Y\right)  $ and choose $\eta>0$. Assume that
$X$ and $Y$ are embedded in a space $Z$ such that $d_{H}^{Z}\left(
X,Y\right)  <\varepsilon+\eta$. Let $A\subset Y$ be such that $Y\subset
A+\left(  \rho+\eta\right)  $ and $\card\left(  A\right)  =N\left(
Y,\rho\right)  $. Let $B\subset X$ be a finite set such that $\card\left(
A\right)  \geq\card\left(  B\right)  $ and each point of $B$ is closer to some
point of $A$ than $\varepsilon+\eta$.
\begin{align*}
X  &  \mathbb{\subset}Y+\left(  \varepsilon+\eta\right) \\
&  \subset\left(  A+\left(  \rho+\eta\right)  \right)  +\left(  \varepsilon
+\eta\right) \\
&  \subset A+\left(  \rho+\varepsilon+2\eta\right) \\
&  \subset\left(  B+\varepsilon+\eta\right)  +\left(  \rho+\varepsilon
+2\eta\right) \\
&  \subset B+\left(  \rho+2\varepsilon+3\eta\right)  \text{.}%
\end{align*}
Whence $N\left(  X,\rho+2\varepsilon+3\eta\right)  \leq\card\left(  B\right)
\leq N\left(  Y,\rho\right)  $. Since the result holds for each $\eta>0$ and
$N\left(  X,\bullet\right)  $ is lower semicontinuous, the claim is proved.

Fix $\rho>0$. Since $\left(  X_{n}\right)  $ is Cauchy, there exists a number
$K$ such that for all $n>K$, $d_{GH}\left(  X_{K},X_{n}\right)  <\frac{\rho
}{3}$. It follows that for all $n>K$, $N\left(  X_{n},\rho\right)  \leq
N\left(  X_{n},\frac{\rho}{3}+2d_{GH}\left(  X_{K},X_{n}\right)  \right)  \leq
N\left(  X_{K},\frac{\rho}{3}\right)  $. Hence
\[
\sup_{n}N\left(  X_{n},\rho\right)  \leq\max\left(  \max_{n=1,...,K}N\left(
X_{n},\rho\right)  ,N\left(  X_{K},\frac{\rho}{3}\right)  \right)
<\infty\text{.}%
\]
In other words, the set of $\left\{  X_{n}|n\in\mathbb{N}\right\}  $ is
uniformly totally bounded, and so, by Gromov's compactness theorem is
relatively compact. Hence, we can extract from $\left(  X_{n}\right)  $ a
converging subsequence. Moreover $\left(  X_{n}\right)  $ is Cauchy, and thus converges.
\end{proof}

\section{Finite spaces}

The subset of $\mathfrak{M}_{F}\subset\mathfrak{M}$ of finite metric spaces is
playing a key role in the study of $\mathfrak{M}$ because it is dense in
$\mathfrak{M}$ and simple enough to be described by mean of matrices.

We define the \emph{codiameter} of finite metric space $X$ as the minimum of
all non zero distances in $X$: $\cdm\left(  X\right)  =\min_{x\neq y\in
X}d^{X}\left(  x,y\right)  $.

A \emph{distance matrix} is symmetric matrix $D=\left(  d_{ij}\right)
_{\substack{1\leq i\leq n\\1\leq j\leq n}}$with $0$'s on the diagonal, and
positive numbers elsewhere, such that for all indices $i,j,k$ we have
$d_{ij}\leq d_{ik}+d_{kj}$. We say that two distance matrices are equivalent,
if we can pass from one to the other by applying the same permutation
simultaneously to its rows and its columns.

We clearly can associate to a distance matrix $D$ of order $n$ the metric
space $X_{D}$ defined by
\begin{align*}
X_{D} &  =\left\{  1,...,n\right\}  \\
d^{X_{D}}\left(  i,j\right)   &  =d_{ij}\text{.}%
\end{align*}
The spaces $X_{D}$ and $X_{D^{\prime}}$ are isometric if and only if $D$ and
$D^{\prime}$ are equivalent. Conversely, given a finite metric space
$X=\left\{  x_{1},...,x_{n}\right\}  $, we can associate to it the distance
matrix, $D=\left(  d^{X}\left(  x_{i},x_{j}\right)  \right)  _{\substack{1\leq
i\leq n\\1\leq j\leq n}}$. Of course $D$ depends on the order in which the
points of $X$ are labeled, but not its class of equivalence. Moreover, two
isometric spaces $X$ and $X^{\prime}$ give the same class of equivalence of
distance matrices. Hence the set $\mathfrak{M}_{n}\subset\mathfrak{M}$ of
those metric spaces with cardinality $n$ is bijectively mapped on the set of
equivalence classes of distance matrices of order $n$. Furthermore the
inequality%
\begin{equation}
d_{GH}\left(  X_{D},X_{D^{\prime}}\right)  \leq\frac{1}{2}\max_{i,j}\left\vert
d_{ij}-d_{ij}^{\prime}\right\vert =\frac{1}{2}\left\Vert D-D^{\prime
}\right\Vert _{\infty}\label{1}%
\end{equation}
follows from Lemma \ref{Ldis}.

Conversely, let $X$ and $Y$ be two finite metric spaces with $n$ elements,
such that $d_{GH}\left(  X,Y\right)  <\frac{1}{2}\cdm\left(  Y\right)  $. Let
$\theta$ be a real number such that $d_{GH}\left(  X,Y\right)  <\theta
<\frac{1}{2}\cdm\left(  Y\right)  $. By definition of $d_{GH}$, there exist a
metric space $Z$, and two subsets $X^{\prime}=\left\{  x_{1},...,x_{n}%
\right\}  $ and $Y^{\prime}=\left\{  y_{1},...,y_{n}\right\}  $ of $Z$, such
that $X$ is isometric to $X^{\prime}$, $Y$ is isometric to $Y^{\prime}$ and
$d_{H}^{Z}\left(  X^{\prime},Y^{\prime}\right)  <\theta$. It follows that for
each $i\in\left\{  1,...,n\right\}  $, there exits $j\in\left\{
1,...,n\right\}  $ such that $d^{Z}\left(  x_{i},y_{j}\right)  <\theta$.
Moreover $j$ is unique: assume on the contrary that two such indices $j_{1}$
and $j_{2}$ exist, then
\[
2\theta<\cdm\left(  Y\right)  \leq d^{Y}\left(  y_{j_{1}},y_{j_{2}}\right)
\leq d^{Z}\left(  y_{j_{1}},x_{i}\right)  +d^{Z}\left(  x_{i},y_{j_{2}%
}\right)  <2\theta\text{.}%
\]
Hence, by changing the labeling of elements of $Y$, we can assume that
$d^{Z}\left(  x_{i},y_{i}\right)  $ is less than $\theta$ for all indices $i$.
Let $D_{X}$ and $D_{Y}$ be the distance matrices of $X$ and $Y$ corresponding
to this order, then
\[
\frac{1}{2}\left\Vert D_{X}-D_{Y}\right\Vert _{\infty}\leq\min_{i,j}\frac
{1}{2}\left(  d^{Z}\left(  x_{i},y_{i}\right)  +d^{Z}\left(  x_{j}%
,y_{j}\right)  \right)  <\theta\text{.}%
\]
Since this holds for all $\theta$ greater than $d_{GH}\left(  X_{D}%
,X_{D^{\prime}}\right)  $, it follows that
\[
\frac{1}{2}\left\Vert D_{X}-D_{Y}\right\Vert _{\infty}\leq d_{GH}\left(
X,Y\right)  \text{.}%
\]
This and (\ref{1}) prove that, if $d_{GH}\left(  X_{D},X_{D^{\prime}}\right)
\leq\frac{1}{2}\cdm\left(  X_{D}\right)  $, then%
\begin{equation}
d_{GH}\left(  X_{D},X_{D^{\prime}}\right)  =\frac{1}{2}\min_{D^{\prime\prime
}\sim D^{\prime}}\left\Vert D-D^{^{\prime\prime}}\right\Vert _{\infty}\text{.}
\label{2}%
\end{equation}
In other words, the bijection between $\mathfrak{M}_{n}$ and the set of
equivalence classes of distance matrix of order $n$ is a local similitude. We
will use this fact to prove the

\begin{lemma}
Let $X\in\mathfrak{M}_{n}$ be a finite metric space with cardinality $n$ and
$\varepsilon$ be a positive number. There exits a ball $B_{0}\subset B\left(
X,\varepsilon\right)  \subset\mathfrak{M}_{n}$, each point of which is totally
anisometric and without triples of collinear points.
\end{lemma}

\begin{proof}
Put $m=\frac{n\left(  n-1\right)  }{2}$. Let $D$ be a distance matrix
associated to $X$. Since the space $\mathcal{D}_{n}$ of distance matrices of
order $n$ is defined by a finite number of linear inequalities ($d_{ij}>0$,
$d_{ij}+d_{jk}\geq d_{ik}$) it is a convex polytope of the set $\mathcal{S}%
_{n}$ of all symmetric matrices with zero on the diagonal, which in turn is
isomorphic to $\mathbb{R}^{m}$ as a vector space. Moreover the distance matrix
with $0$ on the diagonal and $1$ elsewhere clearly belongs to the interior of
$\mathcal{D}_{n}$, that is therefore nonempty. Let $R_{1}$ be the union of the
$\frac{\left(  m-1\right)  m}{2}$ hyperplanes of $\mathcal{S}_{n}$ defined by
the equations $d_{ij}=d_{kl}$ ($1\leq i<j\leq n$, $1\leq k<l\leq n$, $\left(
i,j\right)  <\left(  k,l\right)  $) and $R_{2}$ the union of the $\left(
n-2\right)  m$ hyperplane defined by the equations $d_{ij}+d_{jk}=d_{ik}$
($1\leq i<k\leq n$, $1\leq j\leq n$, $j\neq k$, $j\neq i$). The matrix of a
metric space of cardinality $n$ which is not totally anisometric (resp. admits
a triple of collinear points) should belong to $R_{1}$ (resp. $R_{2}$). Since
$R_{1}\cup R_{2}$ is clearly rare, there exists a ball $B_{1}=B\left(
\Delta,2\eta\right)  $ included in $B\left(  D,\varepsilon\right)  \cap\left(
\mathcal{D}_{n}\backslash\left(  R_{1}\cup R_{2}\right)  \right)  $. Assume
moreover that $\eta<\frac{1}{2}\cdm\left(  X_{\Delta}\right)  $. Let $Y\in
B_{0}\overset{def}{=}B\left(  X_{\Delta},\eta\right)  \subset\mathfrak{M}_{n}%
$. By (\ref{2}), there exists a distance matrix $D_{Y}$ associated to $Y$ such
that $\left\Vert \Delta-D_{Y}\right\Vert =2d_{GH}\left(  X_{\Delta},Y\right)
<2\eta$. Whence $D_{Y}\in B_{1}$ and $Y$ is totally anisometric and without
collinear points.
\end{proof}

\begin{corollary}
\label{C1}Totally anisometric spaces without collinear points are dense in
$\mathfrak{M}$.
\end{corollary}

\section{Basic properties of generic compact metric spaces}

\begin{theorem}
\label{T1}A generic compact metric space is totally discontinuous.
\end{theorem}

\begin{proof}
Let $P_{n}\subset\mathfrak{M}$ be the set of compact metric spaces admitting a
connected component of diameter at least $\frac{1}{n}$. Since $\mathfrak{M}%
_{F}$ is dense in $\mathfrak{M}$, $P_{n}$ has empty interior. The union
$\bigcup_{n\in\mathbb{N}}P_{n}$ is the complementary of the set of totally
discontinuous compact metric spaces, thus we only need to prove that $P_{n}$
is closed. Let $\left(  X_{k}\right)  _{k\in\mathbb{N}}$ be a sequence of
elements of $P_{n}$, converging to $X\in\mathfrak{M}$. Let $C_{k}\subset
X_{k}$ be a closed connected subset whose diameter is at least $\frac{1}{n}$.
By Lemma \ref{LPS}, we can assume without loss of generality that all $X_{k}$
and $X$ are subsets of a compact metric space $Z$ and that $d_{H}^{Z}\left(
X_{k},X\right)  \leq d_{GH}\left(  X_{k},X\right)  +\frac{1}{k}$. Since
$\mathfrak{M}\left(  Z\right)  $ is compact, we can extract from $\left(
C_{k}\right)  $ a converging subsequence. Let $C$ be its limit, it is easy to
see that $C\subset X$. Since the diameter function is continuous, it is clear
that $\diam\left(  C\right)  =\lim\diam\left(  C_{k}\right)  \geq\frac{1}{n}$.
Moreover, it is a well-known fact that the set of connected compact subsets of
$Z$ is closed in $\mathfrak{M}\left(  Z\right)  $ \cite{Michael}. Hence $C$ is
connected and $X$ belongs to $P_{n}$. Thus $P_{n}$ is closed.
\end{proof}

\begin{theorem}
\label{T2}A generic compact metric space is totally anisometric.
\end{theorem}

\begin{proof}
Let $P\subset\mathfrak{M}$ be the set of non totally anisometric space.
\begin{align*}
P  &  =\left\{  X\left\vert
\begin{array}
[c]{l}%
\exists x,y,x^{\prime},y^{\prime}\in X~d\left(  x,y\right)  =d\left(
x^{\prime},y^{\prime}\right)  >0\\
\phantom{\exists x,y,x^{\prime},y^{\prime}\in X\ }\nph{and\ }and~d\left(
x,x^{\prime}\right)  +d\left(  y,y^{\prime}\right)  >0\\
\phantom{\exists x,y,x^{\prime},y^{\prime}\in X\ }\nph{and\ }and~d\left(
x,y^{\prime}\right)  +d\left(  x^{\prime},y\right)  >0
\end{array}
\right.  \right\} \\
&  =\bigcup_{n\in\mathbb{N}}P_{\frac{1}{n}}\text{,}%
\end{align*}
where%
\[
P_{\varepsilon}=\left\{  X\left\vert
\begin{array}
[c]{l}%
\exists x,y,x^{\prime},y^{\prime}\in X~d\left(  x,y\right)  =d\left(
x^{\prime},y^{\prime}\right)  \geq\varepsilon\\
\phantom{\exists x,y,x^{\prime},y^{\prime}\in X\ }\nph{and\ }and~d\left(
x,x^{\prime}\right)  +d\left(  y,y^{\prime}\right)  \geq\varepsilon\\
\phantom{\exists x,y,x^{\prime},y^{\prime}\in X\ }\nph{and\ }and~d\left(
x,y^{\prime}\right)  +d\left(  x^{\prime},y\right)  \geq\varepsilon
\end{array}
\right.  \right\}  \text{.}%
\]
By virtue of Corollary \ref{C1}, it is sufficient to prove that
$P_{\varepsilon}$ is closed. Let $\left(  X_{n}\right)  _{n\in\mathbb{N}}$ be
a sequence of elements of $P_{\varepsilon}$ tending to $X\in\mathfrak{M}$.
There exists a sequence of correspondences $R_{n}$ from $X$ to $X_{n}$ such
that $\dis\left(  R_{n}\right)  $ tends to zero. Let $x_{n},y_{n}%
,x_{n}^{\prime},y_{n}^{\prime}\in X_{n}$ be such that
\begin{align}
d\left(  x_{n,}y_{n}\right)  =d\left(  x_{n}^{\prime},y_{n}^{\prime}\right)
&  \geq\varepsilon\nonumber\\
d\left(  x_{n},x_{n}^{\prime}\right)  +d\left(  y_{n},y_{n}^{\prime}\right)
&  \geq\varepsilon\label{5}\\
d\left(  y_{n},x_{n}^{\prime}\right)  +d\left(  x_{n},y_{n}^{\prime}\right)
&  \geq\varepsilon\text{.}\nonumber
\end{align}
There exists $\tilde{x}_{n},\tilde{x}_{n}^{\prime},\tilde{y}_{n},\tilde{y}%
_{n}^{\prime}\in X$ such that $\tilde{x}_{n}R_{n}x_{n}$, $\tilde{y}_{n}%
R_{n}y_{n}$, $\tilde{x}_{n}^{\prime}R_{n}x_{n}^{\prime}$, and $\tilde{y}%
_{n}^{\prime}R_{n}y_{n}^{\prime}$. Let $\left(  \tilde{x},\tilde{x}^{\prime
},\tilde{y},\tilde{y}^{\prime}\right)  $ be the limit of a converging
subsequence of $\left(  \tilde{x}_{n}\right.  ,\allowbreak\tilde{x}%
_{n}^{\prime},\allowbreak\tilde{y}_{n},\allowbreak\left.  \tilde{y}%
_{n}^{\prime}\right)  $. Since $\dis\left(  R_{n}\right)  $ tends to zero, we
can pass to the limit in (\ref{5}):%
\begin{align*}
d\left(  \tilde{x},\tilde{y}\right)  =d\left(  \tilde{x}^{\prime},\tilde
{y}^{\prime}\right)   &  \geq\varepsilon\\
d\left(  \tilde{x},\tilde{x}^{\prime}\right)  +d\left(  \tilde{y},\tilde
{y}^{\prime}\right)   &  \geq\varepsilon\\
d\left(  \tilde{y},\tilde{x}^{\prime}\right)  +d\left(  \tilde{x},\tilde
{y}^{\prime}\right)   &  \geq\varepsilon\text{.}%
\end{align*}
Hence $X$ belongs to $P_{\varepsilon}$. This completes the proof.
\end{proof}

\begin{theorem}
\label{T3}In a generic compact metric space, three distinct points are never collinear.
\end{theorem}

\begin{proof}
By Corollary \ref{C1}, it is sufficient to prove that
\[
P_{\varepsilon}\overset{def}{=}\left\{  X\in M\left\vert
\begin{array}
[c]{l}%
\exists x,y,z\in X~d\left(  x,y\right)  =d\left(  x,z\right)  +d\left(
z,y\right) \\
\phantom{\exists x,y,z\in X\ }\nph{and\ }and~d\left(  x,z\right)
\geq\varepsilon~and~d\left(  x,z\right)  \geq\varepsilon
\end{array}
\right.  \right\}
\]
is closed. The proof is totally similar to the one of Theorem \ref{T2}.
\end{proof}

\begin{theorem}
\label{T6}A generic compact metric space is perfect.
\end{theorem}

\begin{proof}
Let $P_{n}=\left\{  X\in\mathfrak{M}|\exists x\in X~\forall x^{\prime}\in
X~d\left(  x,x^{\prime}\right)  \in A_{n}\right\}  $, where $A_{n}%
\overset{def}{=}\left\{  0\right\}  \cup\left]  \frac{1}{n},+\infty\right[  $.
The set of non perfect compact metric spaces is the union $\bigcup
_{n\in\mathbb{N}}P_{n}$. It is therefore sufficient to prove that $P_{n}$ is rare.

We claim that the set of perfect compact metric spaces is dense in
$\mathfrak{M}$, and so the interior of $P_{n}$ is empty. Indeed, it is
sufficient to prove that any finite metric space $F$ can be approximated by
perfect spaces. Let $\varepsilon$ be less than the codiameter of $F$ and endow
the product $F_{\varepsilon}\overset{def}{=}F\times\left[  0,\varepsilon
\right]  $ with the distance%
\begin{align*}
d^{F_{\varepsilon}}\left(  \left(  a,s\right)  ,\left(  b,t\right)  \right)
&  =d^{F}\left(  a,b\right)  +s+t\text{, if }a\neq b\\
d^{F_{\varepsilon}}\left(  \left(  a,t\right)  ,\left(  a,s\right)  \right)
&  =\left\vert t-s\right\vert \text{.}%
\end{align*}
It is easy to see that $d^{F_{\varepsilon}}$ is a distance on $F_{\varepsilon
}$, that $F_{\varepsilon}$ is perfect and that $d\left(  F_{\varepsilon
},\varepsilon\right)  \leq\varepsilon$. This proves the claim.

It remains to prove that $P_{n}$ is closed. Let $\left(  X_{k}\right)
_{k\in\mathbb{N}}$ be a sequence of elements of $P_{n}$ converging to
$Y\in\mathfrak{M}$. Let $R_{k}$ be a correspondence between $X_{k}$ and $Y$
such that $\varepsilon_{k}\overset{def}{=}\dis\left(  R_{k}\right)
\leq3d_{GH}\left(  X_{k},Y\right)  $. Let $x_{k}\in X_{k}$ be such that for
all $x^{\prime}\in X_{k}$, $d^{X_{k}}\left(  x_{k},x^{\prime}\right)  \in
A_{n}$. Let $y_{k}\in Y$ correspond to $x_{k}$ by $R_{k}$. By extracting a
suitable subsequence, we can assume that the sequence $\left(  y_{k}\right)
_{k\in\mathbb{N}}$ converges to some point $y\in Y$. Let $y^{\prime}$ be a
point of $Y$ and let $x_{k}^{\prime}\in X_{k}$ correspond to $y^{\prime}$ by
$R_{k}$. We have%
\begin{align*}
d^{Y}\left(  y,y^{\prime}\right)   &  \in\left\{  d^{Y}\left(  y_{k}%
,y^{\prime}\right)  \right\}  +2d^{Y}\left(  y,y_{k}\right) \\
&  \subset\left\{  d^{X_{k}}\left(  x_{k},x_{k}^{\prime}\right)  \right\}
+2\left(  d^{Y}\left(  y,y_{k}\right)  +\varepsilon_{k}\right) \\
&  \subset A_{n}+2\left(  d^{Y}\left(  y,y_{k}\right)  +\varepsilon
_{k}\right)  \text{.}%
\end{align*}
Since the relation holds for arbitrary large $k$, we have $d^{Y}\left(
y,y^{\prime}\right)  \in A_{n}$, whence $Y\in P_{n}$.
\end{proof}

\begin{corollary}
\label{C6}A generic compact metric space is a Cantor space.
\end{corollary}

\begin{proof}
A Brouwer's Theorem (\cite[(7.4)]{Kech}), states that a topological space is a
Cantor space if and only if it is non-empty, perfect, compact, totally
disconnected, and metrizable. Thus the result follows from Theorems \ref{T1}
and \ref{T6}.
\end{proof}

\begin{theorem}
\label{T4} Let $X$ be a generic compact metric space. The set $d\left(
X\right)  \overset{def}{=}\left\{  d^{X}\left(  x,y\right)  |x,y\in X\right\}
$ is homeomorphic to the Cantor set.
\end{theorem}

\begin{proof}
By virtue of the Brouwer's theorem we quote in the proof of Corollary
\ref{C6}, it is sufficient to prove that $d\left(  X\right)  $ is totally
discontinuous and perfect.

Let $P_{a,\varepsilon}=\left\{  X\in\mathfrak{M}|[a,a+\varepsilon]\subset
d\left(  X\right)  \right\}  $. For $\varepsilon>0$, $\mathfrak{M}_{F}\cap
P_{a,\varepsilon}=\emptyset$, whence $P_{a,\varepsilon}$ has empty interior.
The space $X$ admits a non totally discontinuous $d\left(  X\right)  $ if and
only if it belongs to the countable union $\bigcup_{\substack{a\geq
0\\a\in\mathbb{Q}}}\bigcup_{n\in\mathbb{N}}P_{a,\frac{1}{n}}$. Hence we only
have to prove that $P_{a,\varepsilon}$ is closed. Let $\left(  X_{k}\right)
_{k\in\mathbb{N}}$ be a sequence of elements of $P_{a,\varepsilon}$ converging
to $X\in\mathfrak{M}$. Let $R_{k}$ be a correspondence between $X_{k}$ and $X$
such that $\dis\left(  R_{k}\right)  $ tends to zero. Let $b$ be a number of
$\left[  a,a+\varepsilon\right]  $. By hypothesis, there exist $x_{k},y_{k}\in
X_{k}$ such that $d^{X_{k}}\left(  x_{k},y_{k}\right)  =b$. Let $x_{k}%
^{\prime},y_{k}^{\prime}\in X$ correspond by $R_{k}$ to $x_{k}$ and $y_{k}$
respectively, and extract from $\left(  x_{k}^{\prime}\right)  $ and $\left(
y_{k}^{\prime}\right)  $ some subsequences converging to $x$ and $y$
respectively. Since
\begin{align*}
\left\vert b-d\left(  x_{k}^{\prime},y_{k}^{\prime}\right)  \right\vert  &
=\left\vert d\left(  x_{k},y_{k}\right)  -d\left(  x_{k}^{\prime}%
,y_{k}^{\prime}\right)  \right\vert \\
&  \leq\dis\left(  R_{k}\right)  \rightarrow0\text{,}%
\end{align*}
$d\left(  x,y\right)  =b$, and $b\in d\left(  X\right)  $. This holds for all
$b$ in $\left[  a,a+\varepsilon\right]  $, whence $X\in P_{a,\varepsilon}$.
Hence $P_{\alpha,\varepsilon}$ is closed. It follows that for a generic
$X\in\mathfrak{M}$, $d\left(  X\right)  $ is totally discontinuous.

We will now prove that $d\left(  X\right)  $ is perfect. Note that $d\left(
X\right)  =\bigcup_{x_{0}\in X}f_{x_{0}}\left(  X\right)  $, where $f_{x_{0}%
}:X\rightarrow\mathbb{R}$ is the distance function from $x_{0}\in X$. So, it
is sufficient to prove that $f_{x_{0}}\left(  X\right)  $ is perfect. Since by
Theorem \ref{T2} a generic $X\in\mathfrak{M}$ is totally anisometric, the
functions $f_{x_{0}}$ are injective, and so are homomorphisms between $X$ and
$f_{x_{0}}\left(  X\right)  $. Theorem \ref{T6} completes the proof.
\end{proof}

\begin{theorem}
\label{T7}The set of compact metric spaces which contain a $4$-points subspace
that cannot be isometrically imbedded in $\mathbb{\mathbb{R}}^{3}$ is open and
dense in $\mathfrak{M}$. Therefore, a generic compact metric space cannot be
embedded in any Hilbert space.
\end{theorem}

\begin{proof}
It is well-known (see for instance \cite[(10.6.5)]{BergerGI94}) that the
square of the volume of a (possibly degenerated) tetrahedron of $\mathbb{R}%
^{3}$ is given by the following formula (the so-called Cayley-Menger
determinant)
\[
\phi\left(  r,s,t,r^{\prime},s^{\prime},t^{\prime}\right)  =\frac{1}%
{288}\left\vert
\begin{array}
[c]{ccccc}%
0 & 1 & 1 & 1 & 1\\
1 & 0 & t^{2} & s^{2} & r^{\prime2}\\
1 & t^{2} & 0 & r^{2} & s^{\prime2}\\
1 & s^{2} & r^{2} & 0 & t^{\prime2}\\
1 & r^{\prime2} & s^{\prime2} & t^{\prime2} & 0
\end{array}
\right\vert \text{,}%
\]
where $r$, $s$ and $t$ are the lengths of the sides of a one of the faces of
the tetrahedron, and $r^{\prime}$, $s^{\prime}$, $t^{\prime}$ are the lengths
of the edges respectively opposite to the ones of length $r$, $s$, $t$.
Whereas $\phi$ was initially defined only for the sextuples $\left(
r,s,t,r^{\prime},s^{\prime},t^{\prime}\right)  $ which actually correspond to
a tetrahedron, as a polynomial function, it can be extended to $\mathbb{R}%
^{6}$. Given a metric space $A=\left\{  a_{0},a_{1},a_{2},a_{3}\right\}
\in\mathfrak{M}_{4}$, we put%
\begin{align*}
\phi\left(  A\right)   &  \overset{def}{=}\phi\left(  d^{A}\left(  a_{1}%
,a_{2}\right)  ,d^{A}\left(  a_{2},a_{3}\right)  ,d^{A}\left(  a_{3}%
,a_{1}\right)  ,\right. \\
&  \phantom{\phi \left( d^{A}\left( a_{1},a_{2}\right) ,\right.}d^{A}\left(
a_{0},a_{3}\right)  ,\left.  d^{A}\left(  a_{0},a_{1}\right)  ,d^{A}\left(
a_{0},a_{2}\right)  \right)  \text{.}%
\end{align*}
If $\phi\left(  A\right)  <0$, then surely $A$ cannot be isometrically
embedded in $\mathbb{R}^{3}$, nor in any Hilbert space. We will prove that the
set%
\[
P\overset{def}{=}\left\{  X\in\mathfrak{M}\left\vert \exists A\subset
X~\card\left(  A\right)  =4~and~\phi\left(  A\right)  <0\right.  \right\}
\]
is open and dense in $\mathfrak{M}$.

Let $X$ be in $P$ and $A=\left\{  a_{0},a_{1},a_{2},a_{3}\right\}  \subset X$,
such that $\card\left(  A\right)  =4$ and $\phi\left(  A\right)  <0$. Since
$\phi$ is continuous, there exists $\eta>0$ such that, for any $A^{\prime}%
\in\mathfrak{M}_{4}$, if there exists a correspondence between $A$ and
$A^{\prime}$ whose distortion is less than $\eta$, then $\phi\left(
A^{\prime}\right)  <0$. Let $Y\in U\overset{def}{=}B\left(  X,\frac{1}{3}%
\min\left(  \eta,\cdm\left(  A\right)  \right)  \right)  $; there exists a
correspondence $R$ of distortion less than $\min\left(  \eta,\cdm\left(
A\right)  \right)  $ between $X$ and $Y$. Let $A^{\prime}=\left\{
a_{0}^{\prime},a_{1}^{\prime},a_{2}^{\prime},a_{3}^{\prime}\right\}  \subset
Y$ be such that $a_{i}Ra_{i}^{\prime}$, $i=0,\dots,3$. Since $\dis\left(
R\right)  <\cdm\left(  A\right)  $, $A^{\prime}$ is a $4$ points set, and
since $\dis\left(  R\right)  <\eta$, $\phi\left(  A^{\prime}\right)  <0$.
Hence $U\subset P$. It follows that $P$ is open.

Denote by $A_{\varepsilon}$ that the $4$ points space whose distance matrix is%
\[
\left(
\begin{array}
[c]{cccc}%
0 & 2\varepsilon & 2\varepsilon & \varepsilon\\
2\varepsilon & 0 & 2\varepsilon & \varepsilon\\
2\varepsilon & 2\varepsilon & 0 & \varepsilon\\
\varepsilon & \varepsilon & \varepsilon & 0
\end{array}
\right)  \text{.}%
\]
A direct computation shows that $\phi\left(  A_{\varepsilon}\right)
=-\frac{1}{9}\varepsilon^{6}$, and so, any space in which $A_{\varepsilon}$ is
isometrically embedded belongs to $P$. In order to show that $P$ is dense, it
is sufficient to prove that any finite metric space $F$ can be approached by
elements of $P$. Let $F=\left\{  x_{0},\dots,x_{n}\right\}  \in\mathfrak{M}%
_{F}$. For each $\varepsilon>0$ we endow $F_{\varepsilon}\overset{def}%
{=}\left\{  y_{1},y_{2},y_{3},x_{0},x_{1},\dots,x_{n}\right\}  $ with the
distance%
\begin{align*}
d^{F_{\varepsilon}}\left(  x_{i},x_{j}\right)   &  =d^{F}\left(  x_{i}%
,x_{j}\right)  \nph{$d^{F}\left( x_{i},x_{j}\right)$}\hspace{2.5cm}\left(
0\leq i,j\leq n\right) \\
d^{F_{\varepsilon}}\left(  x_{i},y_{j}\right)   &  =d^{F}\left(  x_{i}%
,x_{0}\right)  +\varepsilon
\nph{$d^{F}\left( x_{i},x_{0}\right)+\varepsilon$}\hspace{2.5cm}\left(  0\leq
i\leq n\text{, }1\leq j\leq3\right) \\
d^{F_{\varepsilon}}\left(  y_{i,}y_{j}\right)   &  =2\varepsilon
\nph{$2\varepsilon$}\hspace{2.5cm}\left(  1\leq i,j\leq3\text{, }i\neq
j\right)  \text{.}%
\end{align*}
It is easy to check that $d^{F_{\varepsilon}}$ is a distance on
$F_{\varepsilon}$. Moreover $A_{\varepsilon}$ is embedded (as $\left\{
y_{1},y_{2},y_{3},x_{0}\right\}  $) in $F_{\varepsilon}$, whence
$F_{\varepsilon}\in P$. At last $d_{GH}\left(  F,F_{\varepsilon}\right)
\leq\varepsilon$, whence $P$ is dense in $\mathfrak{M}$.
\end{proof}

\section{\label{SD}Dimensions}

We can associate to a compact metric space several (possibly coinciding)
numbers which all deserve to be called \emph{dimension}. Among the most used,
one distinguishes the\emph{\ topological dimension} ($\dim_{T}$), the
\emph{Hausdorff dimension} ($\dim_{H}$), the \emph{lower }($\dim_{B}$) or
\emph{upper} ($\dim^{B}$) \emph{box dimension}. It is a well-known fact that
for any compact metric space $X$%
\[
\dim_{T}\left(  X\right)  \leq\dim_{H}\left(  X\right)  \leq\dim_{B}\left(
X\right)  \leq\dim^{B}\left(  X\right)  \text{.}%
\]
We refer to \cite{Myjak} or \cite{Falconer} for more details on this subject.
For our purpose, we only need to recall that the upper and lower box (also
called \emph{box-counting}, \emph{fractal} \cite{Myjak}, \emph{entropy}
\cite{Grub}, \emph{capacity}, \emph{Kolmogorov}, \emph{Minkowski}, or
\emph{Minkowski-Bouligand}) dimensions are defined as%
\begin{align*}
\dim^{B}\left(  X\right)   &  =-\limsup_{\varepsilon\rightarrow0}\frac{\log
N\left(  X,\varepsilon\right)  }{\log\varepsilon}\\
\dim_{B}\left(  X\right)   &  =-\liminf_{\varepsilon\rightarrow0}\frac{\log
N\left(  X,\varepsilon\right)  }{\log\varepsilon}%
\end{align*}
(the notation $N\left(  X,\varepsilon\right)  $ is defined in Section
\ref{S2}). If we put%
\[
M\left(  X,\varepsilon\right)  =\max\left\{  \card\left(  F\right)  |F\subset
X~\cdm\left(  F\right)  \geq\varepsilon\right\}  \text{,}%
\]
we may replace $N$ by $M$ in the above formulas. This fact follows from the
inequalities%
\[
N\left(  X,\varepsilon\right)  \leq M\left(  X,\varepsilon\right)  \leq
N\left(  X,\varepsilon/3\right)  \text{,}%
\]
which in turn, follow with a little effort from the definitions of $M$ and $N$
\cite[p. 152]{Grub}.

The generic dimension of some compact subset of some fixed complete metric
space has been studied in \cite{Grub} by P. Gruber. He proved that, given a
complete metric space $X$ such that $\left\{  A\in\mathfrak{M}\left(
X\right)  \left\vert \dim_{B}A\geq\alpha\right.  \right\}  $ is dense in
$\mathfrak{M}\left(  X\right)  $, a generic element of $\mathfrak{M}\left(
X\right)  $ has zero lower box dimension, and an upper box dimension greater
than or equal to $\alpha$. In this section, we transpose his result in the
frame of the Gromov-Hausdorff space.

The result of Gruber is based on the following three lemmas

\begin{lemma}
\label{LG1}\emph{\cite[p. 153]{Grub} }Given a complete metric space $X$ and a
positive number $\varepsilon$, the functions $N\left(  \bullet,\varepsilon
\right)  :\mathfrak{M}\left(  X\right)  \rightarrow\mathbb{N}$ and $M\left(
\bullet,\varepsilon\right)  :\mathfrak{M}\left(  X\right)  \rightarrow
\mathbb{N}$ are respectively lower and upper semi-continuous.
\end{lemma}

\begin{lemma}
\label{LG2}\emph{\cite[p. 20]{Grub2} }Let $B$ be a Baire space. Let
$\alpha_{1}$,$\alpha_{2}$,\dots be positive real constants and $\phi_{1}$,
$\phi_{2}$, \dots be non negative upper-continuous real functions on $B$ such
that $\left\{  x\in B|\phi_{n}\left(  x\right)  =o\left(  \alpha_{n}\right)
\right\}  $ is dense in $B$. Then, for a generic point of $B$, the inequality
$\phi_{n}\left(  x\right)  <\alpha_{n}$ holds for infinitely many $n$.
\end{lemma}

\begin{lemma}
\label{LG3}\emph{\cite[p. 20]{Grub2} }Let $B$ be a Baire space. Let $\beta
_{1}$,$\beta_{2}$,\dots be non-negative real constants and $\psi_{1}$,
$\psi_{2}$, \dots be positive lower-continuous real functions on $B$ such that
$\left\{  x\in B|\beta_{n}=o\left(  \psi_{n}\left(  x\right)  \right)
\right\}  $ is dense in $B$. Then, for a generic point of $B$, the inequality
$\beta_{n}<\psi_{n}\left(  x\right)  $ holds for infinitely many $n$.
\end{lemma}

We first transfer Lemma \ref{LG1} in the Gromov-Hausdorff framework.

\begin{lemma}
\label{L1}Given a positive number $\varepsilon$, the functions $N\left(
\bullet,\varepsilon\right)  :\mathfrak{M}\rightarrow\mathbb{N}$ and $M\left(
\bullet,\varepsilon\right)  :\mathfrak{M}\rightarrow\mathbb{N}$ are
respectively lower and upper semi-continuous.
\end{lemma}

\begin{proof}
Let $\left(  X_{n}\right)  _{n\in\mathbb{N}}$ a sequence of compact metric
spaces converging to $X$ with respect to $d_{GH}$. By Lemma \ref{LPS}, we can
assume without loss of generality that all $X_{n}$ and $X$ are subsets of some
compact metric space $Z$, such that $d_{H}^{Z}\left(  X_{n},X\right)  $ tends
to zero. Hence, we can apply Lemma \ref{LG1} in $Z$, whence $M\left(
X,\varepsilon\right)  \geq\limsup M\left(  X_{n},\varepsilon\right)  $ and
$N\left(  X,\varepsilon\right)  \leq\liminf N\left(  X_{n},\varepsilon\right)
$.
\end{proof}

\begin{theorem}
\label{T5}The lower box dimension of a generic compact metric space is zero,
while its upper box dimension is infinite.
\end{theorem}

\begin{proof}
The proof follows rather closely Gruber's one. Let $\tau>0$. Since
$\mathfrak{M}_{F}$ in included in%
\[
\left\{  X\in\mathfrak{M}\left\vert M\left(  X,\frac{1}{n}\right)  =o\left(
n^{\tau}\right)  \right.  \right\}  \text{,}%
\]
this set is dense in $\mathfrak{M}$. Applying Lemma \ref{LG2}, we obtain that
for a generic $X\in\mathfrak{M}$, the inequality $M\left(  X,\frac{1}%
{n}\right)  <n^{\tau}$ holds for infinitely many $n$, whence
\[
\dim_{B}X\leq\liminf_{n\rightarrow\infty}-\frac{\log M\left(  X,\frac{1}%
{n}\right)  }{\log\frac{1}{n}}\leq\tau\text{.}%
\]
In other words, the set $P_{\tau}\overset{def}{=}\left\{  X\in\mathfrak{M}%
|\dim_{B}X>\tau\right\}  $ is meager, and thus the set
\[
\left\{  X\in M|\dim_{B}X>0\right\}  =\bigcup_{k\in\mathbb{N}}P_{1/k}%
\]
is meager too.

Let $D$ be a positive integer. We claim that set of $D$-dimensional (for any
of the aforementioned notion of dimension) compact set is dense in
$\mathfrak{M}$. Indeed, it is sufficient to prove that any finite metric space
can be approximated by a $D$-dimensional one. Let $F$ be a finite metric
space, let $B\subset\mathbb{R}^{D}$ be a $D$-dimensional ball centered at $0$
whose radius $\varepsilon$ is less than $\cdm\left(  F\right)  $, and endow
the product $F\times B$ with the distance%
\begin{align*}
d^{F\times B}\left(  \left(  a,u\right)  ,\left(  a^{\prime},u^{\prime
}\right)  \right)   &  =d^{F}\left(  a,a^{\prime}\right)  +\left\Vert
u\right\Vert +\left\Vert u^{\prime}\right\Vert \text{, if }a\neq a^{\prime}\\
d^{F\times B}\left(  \left(  a,u\right)  ,\left(  a,u^{\prime}\right)
\right)   &  =\left\Vert u-u^{\prime}\right\Vert \text{.}%
\end{align*}
Its easy to see that $d^{F\times B}$ is actually a distance, and that
$d_{GH}\left(  F,F\times B\right)  \leq d_{H}^{F\times B}\left(
F\times\left\{  0\right\}  ,F\times B\right)  =\varepsilon$. Moreover, as a
disjoint union of $D$-dimensional balls, $F\times B$ is $D$-dimensional. This
proves the claim.

If $\dim_{B}X=D$, then for $n$ large enough $n^{\frac{2D}{3}}<N\left(
X,\frac{1}{n}\right)  $, whence $n^{\frac{D}{3}}=o\left(  n^{\frac{2D}{3}%
}\right)  =o\left(  N\left(  X,\frac{1}{n}\right)  \right)  $. It follows that
the set of $D$-dimensional compact metric spaces is included in%
\[
\left\{  X\in\mathfrak{M}\left\vert n^{\frac{D}{3}}=o\left(  N\left(
X,\frac{1}{n}\right)  \right)  \right.  \right\}  \text{,}%
\]
which is thereby dense in $\mathfrak{M}$. Applying Lemma \ref{LG3} we obtain
that, for a generic metric space $X$, the inequality $N\left(  X,\frac{1}%
{n}\right)  >n^{D/3}$ holds for infinitely many $n$. Hence $\dim^{B}%
X\geq\limsup_{n\rightarrow\infty}-\frac{\log N\left(  X,\frac{1}{n}\right)
}{\log\frac{1}{n}}\geq\frac{D}{3}$. In other words the set $Q_{D}\overset
{def}{=}\left\{  X\in M|\dim^{B}X<\frac{D}{3}\right\}  $ is meager, and thus
the set
\[
\left\{  X\in M|\dim^{B}X<\infty\right\}  =\bigcup_{D\in\mathbb{N}}Q_{D}%
\]
is meager too.
\end{proof}

\bibliographystyle{amsplain}
\bibliography{cg}

\end{document}